\documentclass[11pt,a4paper,reqno]{amsart}
\usepackage{latexsym}
\usepackage{bbm,amssymb,amsthm,amsmath}
\usepackage{graphicx}
\usepackage{multirow}
\usepackage{caption}
\usepackage{bm}
\usepackage{multirow}
\usepackage{mathtools}  
\usepackage{float}
\usepackage{xcolor}

\usepackage{xcolor}

\usepackage{tikz}	
\usetikzlibrary{shapes,arrows}
\usetikzlibrary{positioning}

\calclayout

\theoremstyle{plain}
\newtheorem{theorem}{Theorem}[section]	

\newtheorem{corollary}[theorem]{Corollary}

\theoremstyle{definition}
\newtheorem{definition}[theorem]{Definition}
\newtheorem{example}[theorem]{Example}

\theoremstyle{remark}
\newtheorem{remark}[theorem]{Remark}

\numberwithin{equation}{section}	

\def\R{\mathbb{R}}    
  
\def\N{\mathbb{N}}  

\def\Rd{\mathbb{R}^{d}}
\def\Rr{\mathbb{R}^{r}}

\def\Mat#1{{{\rm M}_{#1}(\C)}}

\def\mnu{{\mib \nu}}

\def\mx{\mathbf{x}}

\def\vf{\mathsf{f}}
\def\vu{\mathsf{u}}
\def\vv{\mathsf{v}}

\def\mA{\mathbf{A}}

\def\mC{\mathbf{C}}

\def\lD{\mathcal{D}}

\def\lH{\mathcal{H}}
\def\lL{\mathcal{L}}
\def\lV{\mathcal{V}}
\def\lW{\mathcal{W}}

\def\dom{\operatorname{dom}}

\def\ker{\operatorname{ker}}

\def\scp#1#2{\langle\, #1 \mid #2 \,\rangle}  
\def\scpp#1#2{\Bigl\langle\, #1 \Bigm| #2 \,\Bigr\rangle}  
\def\iscp#1#2{[\, #1 \mid #2 \,]}  

\def\re{\operatorname{Re}\,}	
\def\im{\operatorname{Im}\,}	
\def\mA{{\bold A}}

\def\mC{{\bold C}}

\def\mx{{\bold x}}

\def\mnu{{\boldsymbol \nu}}

\def\vu{{\sf u}}
\def\vf{{\sf f}}
\def\vv{{\sf v}}

\def\Lb#1{{{\mathrm L}^\infty(#1)}}
\def\Mat#1#2{{{\mathrm M}_{#1}(#2)}}

\def\Op{{\lL}}		

\def\rest#1{_{\displaystyle |_{#1}}}

\def\Cbc#1{{{\mathrm C}^{\infty}_{c}(#1)}}

\def\str{\longrightarrow}

\def\dscon{\relbar\joinrel\rightharpoonup}

\def\Skp#1#2{\langle\, #1 \mid #2 \,\rangle}

\def\Nor#1{\| #1 \|}    

\def\Ld#1{{{\mathrm L}^{2}(#1)}}

\def\fab{{{\mathcal F}(\alpha,\beta;L_0+\widetilde{L}_0, \lH)}}
\def\fabs{{{\mathcal F}_{\alpha,\beta}}}

\def\oi#1#2{\langle#1,#2\rangle}        

\begin{document}

\title[G-convergence of Friedrichs systems revisited]{G-convergence of Friedrichs systems revisited}

\author{K.~Burazin}\address{Kre\v{s}imir Burazin,
	School of Applied Mathematics and Computer Science, J. J. Strossmayer University of Osijek, Trg Ljudevita Gaja 6, 
	31000 Osijek, Croatia
	}	\email{kburazin@mathos.hr}

\author{M.~Erceg}\address{Marko Erceg,
	Department of Mathematics, Faculty of Science, University of Zagreb, Bijeni\v{c}ka cesta 30,
	10000 Zagreb, Croatia
	}\email{maerceg@math.hr}

\author{M.~Waurick}\address{Marcus Waurick,
	Institute for Applied Analysis,
	Faculty of Mathematics and Computer Science,
	TU Bergakademie Freiberg, Pr\"uferstra\ss e 9, 09599 
	Freiberg, Germany
	}\email{marcus.waurick@math.tu-freiberg.de}




\begin{abstract}
We revisit homogenisation theory for Friedrichs systems. In particular, we show that $G$-compactness can be obtained under severely weaker assumptions than in the original work of Burazin and Vrdoljak (2014). In this way we extend the applicability of $G$-compactness results for Friedrichs systems to equations that yield memory effects in the homogenised limit and detour any usage of  compactness techniques  previously employed.
\end{abstract}

\maketitle

\textbf{MSC2020}: 35B27, 35F45, 35M32, 47F05

\textbf{Keywords}: $G$-convergence, $G$-compactness, Friedrichs systems, Memory effects, Homogenisation

\section{Introduction}\label{sec:intro}

The theory of homogenisation is concerned with the macroscopic description of physical phenomena for materials with heterogeneities on the micro- or mesoscale, see, e.g., \cite{Bensoussan1978,Cioranescu1999,Tartar2009} as general references. Usually this is modelled by coefficient functions $a$ given as functions on $\R^d$ for $d$ suitable and the consideration of $a_n(x)\coloneqq a(nx)$ for $n\in \N$. Then, one associates a sequence $(u_n)_n$ defined as the solutions of a sequence of (partial differential) equations involving $(a_n)_n$ as coefficients. The focus is on the limit $n\to\infty$ and, whether a possibly existent (weak) limit $u$ of the sequence $(u_n)_n$ satisfies an equation similar to $u_n$ with $a_n$ replaced by \emph{some} $a_{\text{hom}}$, \emph{the effective} or \emph{homogenised} coefficient. Whether or not such an $a_{\text{hom}}$ exists is a priori unclear. Thus, in order to identify $a_{\text{hom}}$ one shows that at least for a subsequence of $(u_n)_n$, it is possible to show mere existence of \emph{any} $a_{\text{hom}}$, the closer understanding (e.g., unicity etc.) of which then depends on the homogenisation problem at hand. This note is concerned with the existence proof for such a coefficient for a particular class of (abstract partial) differential equations, the class of \emph{Friedrichs systems}. Prior to elaborating on this new development, let us first mention that the notion of positive symmetric systems or Friedrichs systems dates from Kurt Otto Friedrichs \cite{KOF}, who showed that this class of problems encompasses a variety of initial and boundary value problems for various linear partial differential equations of different types. The rebirth of this idea was motivated by a new interest from numerical analysis of differential equations \cite{EGsemel, HMSW, MJensen} and begins with the development  of the theory of abstract Friedrichs systems in Hilbert spaces \cite{EGbook, EGC, ABcpde}. The theory has been extensively studied during the last two decades both from the theoretical and the numerical point of view, see \cite{ABCE, AEM-2017, TBT, BEmjom, BEF, ES22, MDS} and references therein.

Here we shall revisit and generalise a so-called $G$-compactness statement for Friedrichs systems from \cite{BVcpaa}. The generalisation consists of two components: (1) the results will be applicable to abstract Friedrichs systems having the variable coefficient case as a particular special case and --- more importantly --- (2) we will render an assumption made in \cite{BVcpaa} as entirely obsolete. The reason for (2) to be possible is that we realised that some of the equalities in the respective proofs of \cite{BVcpaa} may be replaced by inequalities and that, thus, only lower/upper semi-continuity needs to be used instead of proper continuity statements. The semi-continuity of the decisive expression is then however identified to be a consequence of the Friedrichs systems structure that is underneath anyway. Thus, in passing, we add to the general theory of Friedrichs systems and provide a semi-continuity result of  independent interest (see Theorem \ref{tm:iskp-cont}).

As a consequence of the improvements just sketched, the $G$-compactness result for Friedrichs systems here applies to a wider variety of equations with less properties to control as compared to the result in \cite{BVcpaa}. In particular, some previously used compactness results for the application of results in \cite{BVcpaa} can be detoured, which, for example, enables us to justify the rationale in \cite{Tbis} to find the homogenised equation (with memory effects) in an infinite-dimensional ordinary differential equation. Moreover, in relation to examples this means that for instance Robin boundary conditions are admitted for all domains (particularly unbounded) that allow for a proper formulation in the framework of Friedrichs operators.

In order to frame the results obtained here in a more general context, we briefly present some of the history of $G$-convergence and -compactness as well as related results next. Initially invented for wave and heat type equations, Spagnolo \cite{Spagnolo1967,Spagnolo1976} introduced $G$-convergence as an abstract means to understand micro-oscillations of the (self-adjoint) conductivity matrices without having to resort to periodicity properties of the coefficients, see, e.g., \cite[Chapter 13]{Cioranescu1999} for a list of results. In parallel to Spagnolo, Tartar and Murat \cite{Murat1997, Tartar1975} introduced the concept of $H$-convergence to include the non-self-adjoint case. Quite naturally, they focused on time-independent divergence form problems. Similar concepts have been applied to other equations like stationary elasticity \cite{FM, Tartar1986}, elastic plates \cite{BJV}, and the like. A nonlocal generalisation of both $G$- and $H$-convergence can be found in \cite{W18_NHC,W22_NHC}. 

From an operator theoretic point of view dealing with operator coefficients rather than multiplication operators in the context of so-called evolutionary equations, theorems concerning $G$-compactness results have been developed in \cite{WaurickPhD,W16_H,W16_Gcon}. Most prominently, the results in \cite{W12_HO,Waurick2014} deal with abstract ordinary differential equations in infinite-dimensional Hilbert spaces and the results in \cite{W13_HP,W14_FE,W16_HPDE} provide a corresponding partial differential equations perspective using the full framework of evolutionary equations introduced by Picard \cite{PicPhy}, see also \cite{STW22} for an accessible monograph introducing the necessary results in considerable detail. The reader is particularly referred to \cite[Chapters 13 and 14]{STW22}. In a more particular setting appropriately generalising stochastic two-scale convergence in the mean, we refer to \cite{NVW}, where the focus is slightly shifted and more assumptions on the oscillatory part are imposed.

In relation to the mentioned work, the present results for Friedrichs systems can be rephrased as $G$-compactness results for the perturbations of $m$-accretive operators. Indeed, in view of the work in \cite{WW18,WW20,PT23}, abstract Friedrichs systems can also be viewed as the theory of finding $m$-accretive extensions of certain potentially unbounded operators. In a nutshell, the main result of the present manuscript confirms that given an $m$-accretive operator $A$, which itself is an extension of a skew-symmetric operator,  and a bounded sequence $(C_n)_n$ of strictly $m$-accretive bounded linear operators, there exists a subsequence and a bounded linear operator $C$ satisfying similar bounds as the $C_n$ do, such that $(C_{n_k}+A)^{-1} \to (C+A)^{-1}$ in the weak operator topology; see Corollary \ref{cor:sksa} below. This is remarkable in as much as there is neither a compactness condition nor a spectral assumption like (skew-)selfadjointness for (the resolvent of) $A$ needed; this contrasts and complements previously known results (see, e.g., \cite[Lemma 4.5]{EGW17_D2N} or \cite[Lemma 14.1.2]{STW22}). Furthermore, note that the perspective presented here revisits the standpoint from \cite{BVcpaa} of looking at some operator topology convergence of $(C_{n_k}+A)^{-1}$ rather than at the convergence of $C_{n_k}$, which is the common starting point of the results for evolutionary equations mentioned above.

Next, we summarise the course of the present manuscript. In Section \ref{sec:abstractFO} we recall the concept of abstract Friedrichs operators and the corresponding concept of $G$-convergence. This section also contains our main result, Theorem \ref{thm:Gconv}, and the condition (K1) from the previously known $G$-compactness statement we were able to entirely dispense with. The proper replacement for (K1) and the corresponding proof that said replacement is always satisfied in the general situation of abstract Friedrichs operators is presented in Section \ref{sec:inequalitybdyop}. Section \ref{sec:proof-Gconv} is devoted to the proof of our main theorem, which in passing also provides the full argument for the complex Hilbert space case, which was omitted for simplicity in \cite{BVcpaa}. In the final Section \ref{sec:classFOs} we apply our abstract findings to classical Friedrichs operators and provide some examples.

\section{G-convergence and abstract Friedrichs operators}\label{sec:abstractFO}

\subsection{Abstract Friedrichs operators}

The abstract Hilbert space formalism for Friedrichs systems which we 
study in this paper was introduced 
and developed in \cite{EGC, ABcpde} for real vector spaces, while 
the required differences for complex vector spaces have been 
supplemented more recently in \cite{ABCE}. 
Here we present the definition in the form given in
\cite[Definition 1]{AEM-2017}.

\begin{definition}\label{def:abstractFO}
A (densely defined) linear operator $T_0$ on a complex Hilbert space $\lH$
(a scalar product is denoted by
$\scp\cdot\cdot$, which we take to be anti-linear in the second entry)
is called an \emph{abstract Friedrichs operator} if there exists a
(densely defined) linear operator $\widetilde{T}_0$ on $\lH$ with the following properties:
\begin{itemize}
 \item[(T1)] $T_0$ and $\widetilde{T}_0$ have a common domain $\lD$, i.e.~$\dom T_0=\dom\widetilde{T}_0=\lD$,
 which is dense in $\lH$, satisfying
 \[
 \scp{T_0\phi}\psi \;=\; \scp\phi{\widetilde T_0\psi} \;, \qquad \phi,\psi\in\mathcal{D} \,;
 \]
 \item[(T2)] there is a constant $\lambda>0$ for which
 \[
 \|(T_0+\widetilde{T}_0)\phi\| \;\leqslant\; 2\lambda\|\phi\| \;, \qquad \phi\in\mathcal{D} \,;
 \]
 \item[(T3)] there exists a constant $\mu>0$ such that
 \[
 \scp{(T_0+\widetilde{T}_0)\phi}\phi \;\geqslant\; 2\mu \|\phi\|^2 \;, \qquad \phi\in\mathcal{D} \,.
 \]
\end{itemize}
The pair $(T_0,\widetilde{T}_0)$ is referred to as a \emph{joint pair of abstract Friedrichs operators}
(the definition is indeed symmetric in $T_0$ and $\widetilde{T}_0$).
\end{definition}

Before moving to the main topic of the paper, 
let us briefly recall the essential properties 
of (joint pairs of) abstract Friedrichs operators, 
which we summarise in the form of a theorem.
At the same time, we introduce the notation that is used throughout the paper.
The presentation is made in two steps: first we deal with the consequences 
of conditions (T1)--(T2), and then we emphasise the additional structure implied by condition (T3). 

\begin{theorem}\label{thm:abstractFO-prop}
Let a pair of linear operators $(T_0,\widetilde{T}_0)$ on $\lH$ satisfy {\rm (T1)} and {\rm (T2)}. Then the following holds.
\begin{enumerate}
\item[i)] $T_0\subseteq \widetilde{T}_0^*=:T_1$ and $\widetilde{T}_0\subseteq T_0^*=:\widetilde{T}_1$, where 
$\widetilde{T}_0^*$ and $T_0^*$ are adjoints of $\widetilde{T}_0$ and $T_0$, respectively.
\item[ii)] The pair of closures $(\overline{T}_0,\overline{\widetilde{T}}_0)$ satisfies {\rm (T1)--(T2)} with the same constant $\lambda$.
\item[iii)] $\dom \overline{T}_0=\dom\overline{\widetilde{T}}_0=:\lW_0$ and $\dom T_1=\dom\widetilde{T}_1=:\lW$.
\item[iv)] The graph norms $\|\cdot\|_{T_1}:=\|\cdot\|+\|T_1\cdot\|$ and $\|\cdot\|_{\widetilde T_1}
	:=\|\cdot\|+\|\widetilde T_1\cdot\|$ are equivalent, $(\lW,\|\,\cdot\,\|_{T_1})$ is a Hilbert space (the \emph{graph space})
	and $\lW_0$ is a closed subspace containing $\lD$.
\item[v)] The linear operator $\overline{T_0+\widetilde{T}_0}$ is everywhere defined, bounded and self-adjoint on 
$\lH$ such that on $\lW$ it coincides with $T_1+\widetilde{T}_1$.
\item[vi)] The sesquilinear map 
\begin{equation}\label{eq:D}
\iscp uv \;:=\; \scp{T_1u}{v} 
- \scp{u}{\widetilde{T}_1v} \;,
\quad u,v\in\lW \,, 
\end{equation}
is an indefinite inner product on $\lW$ and we have $\lW^{[\perp]}=\lW_0$ and $\lW_0^{[\perp]}=\lW$, where
the $\iscp\cdot\cdot$-orthogonal complement of a set $S\subseteq \lW$
is defined by
\begin{equation*} 
S^{[\perp]} := \bigl\{u\in \lW : (\forall v\in S) \quad 
\iscp uv = 0\bigr\}
\end{equation*}
and it is closed in $\lW$.
\end{enumerate}

Assume, in addition, {\rm (T3)}, i.e.~$(T_0,\widetilde{T}_0)$ is 
a joint pair of abstract Friedrichs operators. Then
\begin{itemize}
\item[vii)] $(\overline{T}_0,\overline{\widetilde{T}}_0)$ satisfies {\rm (T3)} with the same constant $\mu$.
\item[viii)] $\overline{T_0+\widetilde{T}_0}$ is coercive.
\item[ix)] We have
\begin{equation}\label{eq:decomposition}
\lW \;=\; \lW_0 \dotplus \ker T_1 \dotplus \ker\widetilde T_1 \;,
\end{equation}
where the sums are direct and all spaces on the right-hand side 
are pairwise $\iscp{\cdot}{\cdot}$-orthogonal. 
Moreover, the (non-orthogonal) linear projections
\begin{equation}\label{eq:projections}
p_\mathrm{k} : \lW \to \ker T_1 \quad \hbox{and} \quad
p_\mathrm{\tilde k}:\lW\to \ker \widetilde{T}_1
\end{equation}
are continuous as maps $(\lW,\|\cdot\|_{T_1})\to (\lH,\|\cdot\|)$, i.e.~$p_\mathrm{k}, p_\mathrm{\tilde k}\in\lL(\lW,\lH)$.
\item[x)] If $\lV$ is a subspace of the graph space $\lW$
satisfying condition {\rm (V)}: $\lV=\widetilde{\lV}^{[\perp]}$, where $\widetilde{\lV}:=\lV^{[\perp]}$,
and
\begin{equation}
\begin{aligned}
&(\forall u\in \lV) \qquad \iscp uu \;\geq\; 0 \,, \\
&(\forall v\in \widetilde\lV) \qquad \iscp{v}{v} \;\leq\; 0\,,
\end{aligned}
\tag{V1}
\end{equation}
then $T_1|_\lV:\lV\to\lH$ and $\widetilde{T}_1|_{\widetilde{\lV}}:\widetilde{\lV}\to\lH$ are bijective,
i.e.~isomorphisms when we equip their domains with the 
graph topology, and for every $u\in\lV$ the following estimate holds:
\begin{equation}\label{eq:apriori}
	\|u\|_{T_1} \leq \Bigl(1+\frac{1}{\mu}\Bigr) \|T_1 u\| \,.
\end{equation}
The same estimate holds for $\widetilde{T}_1$ and $\widetilde{\lV}$ replacing $T_1$ and $\lV$, respectively.
\item[xi)] Let $\lV$ be a closed subspace of $\lW$ such that $\lW_0\subseteq\lV$.
Then $T_1|_\lV:\lV\to\lH$ is bijective if and only if $\lV\dotplus\ker T_1 = \lW$.
\end{itemize}
\end{theorem}

The statements i)--iv), vii) and viii) follow easily from the corresponding
assumptions (cf.~\cite{AEM-2017, EGC}).
The claims v), vi) and x) are already 
argued in the first paper on abstract Friedrichs operators \cite{EGC} for real vector spaces
(see sections 2 and 3 there), while in \cite{ABCE} 
the arguments are repeated in the complex setting.
The decomposition given in ix) is derived in \cite[Theorem 3.1]{ES22},
while for additional claims on projectors we refer to the proof of Lemma 3.5 in the 
aforementioned reference. 
Finally, the part xi) is obtained in \cite[Lemma 3.10]{ES22}.

\begin{remark}\label{rem:Vcond}
Let us note that any subspace $\lV$ of $\lW$ that satisfies assumption {\rm (V)} from part  x) is, by part vi),  closed and contains $\lW_0$, and thus satisfies assumptions from xi). Then clearly we also have $\lV\dotplus\ker T_1 = \lW$. However, not all bijections described in part xi) are of the form
given in part x), i.e.~(V1) is only a sufficient condition. 
In this paper, we will focus only on those given in part x), as we shall make use of the
a priori estimate \eqref{eq:apriori}.
\end{remark}


\begin{remark}\label{rem:Wweak}
The graph space $\lW$ is indeed a Hilbert space when equipped with the graph inner product
$\scp{\cdot}{\cdot}_{T_1}:=\scp{\cdot}{\cdot}+\scp{T_1\,\cdot}{T_1\,\cdot}$.
Then the notion of weak convergence in $\lW$ can be characterised in the following way \cite[Theorem 3]{BVcpaa}:
a sequence $(u_n)$ converges weakly to $u$ in $\lW$ if and only if 
\begin{align*}
& u_n\dscon u \quad \hbox{in $\lH$}\,, \\
& T_1 u_n\dscon T_1 u \quad \hbox{in $\lH$}\,,
\end{align*}
where here, as well as in the rest of the paper, by $\dscon$ we denote the weak convergence in the corresponding space. 

It should be noted that here (for simplicity) the graph norm is not induced by the graph inner product, 
but it is equivalent with the corresponding induced norm. This is also the reason why 
the a priori estimate \eqref{eq:apriori} differs from the one given in \cite[Theorem 1]{BVcpaa}. 
\end{remark}
 
\begin{remark}\label{rem:VtildeV-existence}
In \cite[Theorem 9]{AEM-2017} it is shown that the conditions $\lV=\widetilde{\lV}^{[\perp]}$ and $\widetilde\lV=\lV^{[\perp]}$,
appearing in the part x) above, are equivalent to the fact that the restrictions 
$T_1|_\lV:\lV\to\lH$ and $\widetilde{T}_1|_{\widetilde{\lV}}:\widetilde{\lV}\to\lH$ are mutually adjoint. 
Moreover, it is known that for any joint pair of abstract Friedrichs operators there exists a 
subspace $\lV$ of $\lW$ such that the condition (V) is satisfied (see \cite[Theorem 13(i)]{AEM-2017}), 
i.e.~there exists at least one 
bijective realisation of the corresponding operators.
\end{remark}

\begin{remark}
Note that in the part ix) of the previous theorem we can also consider the graph norm 
in the codomain of projections, since the graph norm and the (standard) 
norm are equivalent on the kernels $\ker T_1$ and $\ker\widetilde{T}_1$.
\end{remark}

We shall now describe the main goal of the manuscript.
Let us assume that we are given a sequence $(T_{0,n},\widetilde{T}_{0,n})$ of joint pairs of 
abstract Friedrichs operators and let us denote by $(T_n,\widetilde{T}_n)$ a pair of bijective 
realisations given by Theorem \ref{thm:abstractFO-prop}(x) (see also Remark \ref{rem:VtildeV-existence}).
Then for any $f\in\lH$ there is a unique $u_n\in \dom T_n$ such that $T_n u_n=f$.
The question of (G-)convergence of $(T_n)_n$ arises naturally: Does
there exist a linear operator $T$ such that for all $f\in\lH$ the sequence of solutions
$(u_n)_n$ converges (in some sense) to a solution $u$ of $Tu=f$?

A first immediate answer to this question roots in the following compactness result.

\begin{theorem}[{{{see, e.g., \cite[Theorem 5.1]{W18_NHC}}}}]\label{thm:seqcomp} Let $\lH_0,\lH_1$ be Hilbert spaces, and denote by $\lL(\lH_0,\lH_1)$ the space of bounded linear operators from $\lH_0$ to $\lH_1$. Then
\[
B_{\lL(\lH_0,\lH_1)} \coloneqq \{ T\in \lL(\lH_0,\lH_1); \|T\|\leq 1\}
\]
is compact under the weak operator topology; it is metrisable and, in particular, sequentially compact, if both $\lH_0$ and $\lH_1$ are separable.
\end{theorem}

However, on this level of generality, assumed that there is an a priori bound on $(T_n^{-1})_n$ uniformly in $n$; a subsequence $(T_{n_k}^{-1})_k$ can be chosen and an operator $S\in \lL(\lH)$ can be found such that $T_{n_k}^{-1}f\to Sf$ weakly in $\lH$ for all $f\in \lH$. Whether or not $S$ is suitably invertible with an inverse stemming from an abstract Friedrichs operator is, however, entirely unclear. One of the challenges one might be confronted with here is that 
the graph spaces of the $T_n$ may depend on $n$ and thus, so may $\dom T_n$. Hence, it is not clear in which sense and in which (optimal) space 
to study the convergence of $(u_n)$ towards some/any $u$.
Therefore, a reasonable requirement is that the graph space and the domain are independent of $n$. In this paper we chose to work precisely with one such class of abstract Friedrichs operators, which is 
broad enough to cover many interesting examples. In particular, the class studied in \cite{BVcpaa} is fully contained in our anaylsis. 
In the following subsection we introduce and study the mentioned class of operators.

\subsection{G-convergence of a class of abstract Friedrichs operators}\label{subsec: Gkonv}

In this part we study a class of abstract Friedrichs operators for which we introduce 
the corresponding notion of G-convergence. This subsection is concluded with the main result of the paper, 
the compactness result for G-convergence of the class of abstract Friedrichs operators.

Let $(L_0,\widetilde{L}_0)$ be a pair of linear operators on $\lH$ satisfying (T1) and (T2).
We use the same notation as introduced in Definition \ref{def:abstractFO} and Theorem \ref{thm:abstractFO-prop}(i)--(vi).
However, to avoid confusion in later discussions, we shall denote the parameter appearing in condition (T2) with 
the subscript $L$, i.e.~$\lambda_L$.

For given $\beta>\alpha>0$, let us define a set $\fab$ 
of bounded linear operators on $\lH$ as follows: $C\in\lL(\lH)$ is in
$\fab$ 
if for any $\varphi\in\lH$ we have
\begin{align}
\re\scpp{\Bigl(C+\frac{1}{2}(\overline{L_0+\widetilde{L}_0})\Bigr)\varphi}{\varphi} &\geq \alpha \|\varphi\|^2 \,,
	\label{eq:Fab_coercive}\\
\re\scpp{\Bigl(C+\frac{1}{2}(\overline{L_0+\widetilde{L}_0})\Bigr)\varphi}{\varphi} &\geq \frac{1}{\beta} \Bigl\|\Bigl(C+\frac{1}{2}(\overline{L_0+\widetilde{L}_0})\Bigr)\varphi\Bigr\|^2 \,, \label{eq:Fab_bdd}
\end{align}
where $\re$ denotes the real part of a complex number (which is of course redudant if the spaces are real).
When it is clear from the context what is the underlying Hilbert space $\lH$ and which operator $L_0+\widetilde{L}_0$ is used, we shall use a shorter notation:
$\fabs=\fab$.

\begin{remark}
	By a density argument it suffices to require that inequalities (\ref{eq:Fab_coercive}) and (\ref{eq:Fab_bdd}) describing $\fab$ are valid for $\varphi\in\lD$, and if this is the case, then we can actually replace $\overline{L_0+\widetilde{L}_0}$ by $L_0+\widetilde{L}_0$. Also note that (\ref{eq:Fab_coercive}) and (\ref{eq:Fab_bdd}) are valid if we take $\varphi\in\lW$ and $L_1+\widetilde{L}_1$ instead of $\overline{L_0+\widetilde{L}_0}$, as $L_1+\widetilde{L}_1\subseteq \overline{L_0+\widetilde{L}_0}$; where for notational convenience, $L_1$ and $\widetilde{L}_1$ is defined in parallel to $T_1$ and $\widetilde{T}_1$, that is, $\widetilde{L}_1\coloneqq L_0^*$ and $L_1\coloneqq {\widetilde{L}}_0^*$.
\end{remark}

For the above choice of $(L_0,\widetilde{L}_0)$, let us take an arbitrary $C\in\fabs$. From \eqref{eq:Fab_bdd}, using the Schwarz--Cauchy--Bunjakovski
inequality, we have
$$
\Bigl\|\Bigl(C+\frac{1}{2}(\overline{L_0+\widetilde{L}_0})\Bigr)\varphi\Bigr\| \leq \beta\|\varphi\| \,,
$$ 
which together with (T2) implies that for any $\varphi\in\lH$
\begin{equation*}
\|C\varphi\| \leq (\beta+\lambda_L)\|\varphi\| \;.
\end{equation*}
Hence, operators in $\fabs$ are uniformly bounded:
\begin{equation}\label{eq:Fab_Cbdd}
\sup_{C\in\fabs} \|C\|_{\lL(\lH)} \leq \beta+\lambda_L \;.
\end{equation}

For an arbitrary $C\in\fabs$ let us define operators
\begin{equation}\label{eq:LplusC}
T_0:=L_0+C \qquad \hbox{and} \qquad \widetilde{T}_0:=\widetilde{L}_0+C^* \;.
\end{equation}
Since $C$ is bounded, it is clear that both $T_0$ and $\widetilde{T}_0$ are densely defined, with
$\dom T_0 = \dom \widetilde{T}_0 = \dom L_0 = \lD$.
Moreover, it is easy to verify that $(T_0,\widetilde{T}_0)$ is a joint pair of abstract Friedrichs operators. 
Indeed, for any $\varphi,\psi\in\lD$ we have
$$
\scp{T_0\varphi}{\psi} = \scp{L_0\varphi}{\psi} + \scp{C\varphi}{\psi}
	= \scp{\varphi}{\widetilde{L}_0\psi} + \scp{\varphi}{C^*\psi} = \scp{\varphi}{\widetilde{T}_0\psi} \,,
$$
implying that condition (T1) is fulfilled. Since $T_0+\widetilde{T}_0$ is a restriction of the sum of two (bounded and) 
self-adjoint operators $\overline{L_0+\widetilde{L}_0}$ and $C+C^*$, condition (T3) is directly implied by 
\eqref{eq:Fab_coercive} and the corresponding constant, denoted by $\mu_T$, agrees with $\alpha$.
Finally, using \eqref{eq:Fab_Cbdd} and the fact that $\|C^*\|_{\lL(\lH)}=\|C\|_{\lL(\lH)}$, we have
$$
\|(T_0+\widetilde{T}_0)\varphi\| \leq 2(2\lambda_L+\beta)\|\varphi\| \,. 
$$

Note also that 
$$
T_1:=\widetilde{T}_0^*=L_1+C \,, \quad \widetilde{T}_1 :=T_0^*=\widetilde{L}_1+C^* \,,
$$
and thus $\dom T_1=\dom \widetilde{T}_1=\dom L_1=\lW$.
Furthemore, it is clear that the indefinite inner product \eqref{eq:D} is independent of $C\in\fabs$:
\begin{equation}\label{eq:D_L}
\iscp uv \;=\; \scp{L_1u}{v} 
- \scp{u}{\widetilde{L}_1v} \;,
\quad u,v\in\lW \,.
\end{equation}
That in particular implies that pairs of subspaces $(\lV,\widetilde{\lV})$ from Theorem \ref{thm:abstractFO-prop}(x)
are independent of $C$, since all conditions describing them are given only in terms of the indefinite inner product. 
In order to simplify notation, for a given $\lV$ we shall denote $L_1|_\lV$ by $L$, and the same for $T_1$.

Note that the corresponding parameters $\lambda_T$ and $\mu_T$ appearing in conditions (T2) and (T3), repectively, 
are both independent of $C$, i.e.~they depend only on $\lambda_L$, $\alpha$ and $\beta$. More precisely, 
from the above we have $\lambda_T=2\lambda_L+ \beta$ and $\mu_T=\alpha$.
The former implies that the graph norms $\|\cdot\|_{T_1}$ and $\|\cdot\|_{L_1}$ are equivalent on $\lW$ with constants not depending on $C\in\fabs$, while the latter implies that the a priori estimate \eqref{eq:apriori} 
is also uniform with respect to $C\in\fabs$. 
Thus, the overall conclusion is that there exists $\gamma=\gamma(\lambda_L,\alpha,\beta)>0$, independent of $C\in\fabs$,
such that for any pair of subspaces $(\lV,\widetilde{\lV})$ satisfying assumptions of 
Theorem \ref{thm:abstractFO-prop}(x) the following a priori estimate holds
\begin{equation}\label{eq:Fab_apriori}
	\|u\|_{L_1} \leq \gamma \|(L_1+C)u\| \,, \quad u\in\lV \,,
\end{equation} 
and analogously for $\widetilde{L}_1$ and $\widetilde{\lV}$.

We are now ready to introduce the notion of G-convergence for a sequence of operators of the form
\eqref{eq:LplusC}. 

\begin{definition}[{\bf G-convergence for Friedrichs operators}]\label{def:Gconv}
Let $(L_0,\widetilde{L}_0)$ be a pair of linear operators satisfying {\rm (T1)} and {\rm (T2)} on a complex Hilbert space $\lH$
and let $\lV$ be a subspace of $\lW$ satisfying {\rm (V)} (see Theorem \ref{thm:abstractFO-prop}(x)). 
We define $L:=L_1|_\lV$.

For a sequence $(C_n)_n$ in $\fab$ we say that the sequence of isomorphisms
$$
T_n:=L+C_n:\lV\to\lH
$$
\emph{G-converges} to an isomorphism $T:=L+C:\lV\to\lH$, for some 
$C\in\fab$, if the (bounded) inverse operators
$T_n^{-1}:\lH\to\lV$ converge in the weak sense:
\begin{equation*}
(\forall f\in\lH) \qquad T_n^{-1}f\dscon T^{-1}f \quad \hbox{in $\lW$} \;.
\end{equation*}
\end{definition}

The compactness of G-convergence was shown in \cite{BVcpaa} for a specific choice of $L_0$ 
(first order differential operator with constant coefficients) and under an additional assumption that
(for fixed $L_0$ and $\lV$ satisfying (V)), the family $\fabs$ has the following property:
\begin{itemize}
	\item [(K1)] for every sequence $C_n\in \fabs$, and every $f\in \lH$, the sequence $(u_n)$ 
	in $\lV$ defined by $u_n:=(L + C_n)^{-1}f$ satisfies the following: 
	if $(u_n)$ weakly converges to $u$ in $\lW$, then also
	$$
	\iscp{u_n}{u_n} \str \iscp{u}{u}\,.
	$$
	
\end{itemize}
By using Theorem \ref{tm:iskp-cont} given below we prove that this assumption is redundant even for the above (abstract)
class of operators.

\begin{theorem}[\textbf{Compactness of G-convergence}]\label{thm:Gconv}
Using the notation of the previous definition, let $\lH$ be a separable Hilbert space and let $L_0$ and $\lV$ be fixed. 
For any sequence $(C_n)_n$ in $\fab$ there exists a subsequence of $T_n:=L+C_n$ which G-converges 
to $T:=L+C$ with $C\in\fab$.
\end{theorem}

Notice that in the previous theorem for the first time we have an additional assumption on $\lH$:
separability. It is a technical requirement that allows the application of the Cantor diagonal procedure, which in turn merely describes the fact that bounded subsets of bounded linear operators on $\lH$ are relatively sequentially compact under the weak operator topology, see Theorem \ref{thm:seqcomp} above.

The proof of Theorem \ref{thm:Gconv} we postpone to Section \ref{sec:proof-Gconv}. 
Before, in the following section, we deal with a suitable replacement of condition (K1) and show that this replacement (in terms of an inequality) is always satisfied.

We close this section by addressing an immediate consequence of Theorem \ref{thm:Gconv} for a special case. The upshot of considering this particular situation is that its formulation is fairly independent of the notion of (abstract) Friedrichs operators. This corollary may thus be used as a first approximation of the more general Theorem \ref{thm:Gconv}. For this, we recall a linear operator $A$ on a Hilbert space $\lH$ is \emph{$m$-accretive}, if $A$ is densely defined and closed with
\[
   \re \scp{A\phi}\phi \geq 0 \text{ and }   \re \scp{A^*\psi}\psi \geq 0
\]for all $\phi\in \dom(A)$, $\psi\in \dom(A^*)$.\footnote{There are many equivalent ways to define `m-accretive'. We chose this definition as it is clearly symmetric in $A$ and $A^*$.}

\begin{corollary}\label{cor:sksa} Let $\lH$ be a separable Hilbert space, $A_0\colon \dom(A_0)\subseteq \lH\to \lH$ be closed and skew-symmetric. Let $A$ be  $m$-accretive with $A_0\subseteq A\subseteq -A_0^*$.

If $(C_n)_n$ is a sequence in $\lL(\lH)$ with
\[
   \re \scp{C_n \phi}\phi \geq \max\Bigl\{\alpha\|\phi\|^2,\frac{1}{\beta}\|C_n\phi\|^2\Bigr\}\quad(\phi\in \lH),
\]then there exists a subsequence $(C_{n_k})_k$ and $C\in \lL(\lH)$ satisfying the same inequalities as $C_n$ such that
\[
   (C_{n_k}+A)^{-1}\to (C+A)^{-1}
\]in the weak operator topology.
\end{corollary}
\begin{proof}
The assumptions of Theorem \ref{thm:Gconv} are satisfied with $L_0\coloneqq A_0$ and $\widetilde{L}_0=-A_0$ (note that $L_0+\widetilde{L}_0=0$), and $\lV=\dom(A)$; also recall Remark \ref{rem:VtildeV-existence} for the condition (V) keeping in mind the accretivity of $A$ (and $A^*\subseteq A_0^*$).
\end{proof}

\section{An inequality for the boundary operator}\label{sec:inequalitybdyop}

\begin{theorem}\label{tm:iskp-cont}
Let $(T_0,\widetilde{T}_0)$ be a joint pair of abstract Friedrichs operators on $\lH$
and let $\lV$ be a subspace of $\lW$ satisfying condition {\rm (V)}.

Then for any weakly converging sequences $(u_n)_n$ in $\lV$ 
and $(v_n)_n$ in $\widetilde{\lV}:=\lV^{[\perp]}$ we have
\begin{align*}
	\liminf_{n\to\infty} \,\iscp{u_n}{u_n} &\geq \iscp{u}{u} \;, \\
	\limsup_{n\to\infty} \,\iscp{v_n}{v_n} &\leq \iscp{v}{v} \;,
\end{align*}
where $u\in \lV$ and $v\in\widetilde{\lV}$ are the weak limits of 
$(u_n)_n$ and $(v_n)_n$, respectively, i.e.~$u_n\overset{\lW}{\dscon} u$
and $v_n\overset{\lW}{\dscon} v$.
\end{theorem}

\begin{proof}
\noindent \textbf{Step I.} By Theorem \ref{thm:abstractFO-prop}(x),(xi) we have 
$\lV\dot{+}\ker T_1 =\lW$. Thus, for any $\tilde\nu\in\ker \widetilde{T}_1$
there exist $u\in\lV$ and $\nu\in\ker T_1$ such that
$u+\nu = \tilde \nu$. This implies $-\nu+\tilde\nu =u \in\lV$. Hence, recalling
the decomposition \eqref{eq:decomposition} we have $p_\mathrm{\tilde k}(u)=\tilde \nu$,
implying $p_\mathrm{\tilde k}(\lV)=\ker \widetilde{T}_1$.

Analogously we have $p_\mathrm{k}(\widetilde{\lV})=\ker T_1$.
\smallskip

\noindent \textbf{Step II.} Now we will focus on the claim for $\lV$, while 
the main differences in the argument for the other one will be addressed in the 
final step of the proof. 

We define an operator $\widetilde{S}:\ker\widetilde{T}_1 \to\lH$ with
\begin{equation}\label{eq:opS}
	\widetilde{S}(p_\mathrm{\tilde k}(u)) := p_\mathrm{k}(u) \;, \quad u\in\lV\,,
\end{equation}
where $p_\mathrm{k}$ and $p_\mathrm{\tilde k}$ are given by \eqref{eq:projections}.

To show that $\widetilde{S}$ is a well-defined function on its domain $p_{\tilde k}(\lV)=\ker \widetilde{T}_1$ 
(here the previous step is used), it is 
sufficient to prove that for any $u,u'\in\lV$ such that
$p_\mathrm{\tilde k}(u)=p_\mathrm{\tilde k}(u')$ we have 
$p_\mathrm{k}(u)=p_\mathrm{k}(u')$. By \eqref{eq:decomposition} we have
\begin{align*}
u &= u_0 + p_\mathrm{k}(u) + p_\mathrm{\tilde k}(u) \\
u' &= u_0' + p_\mathrm{k}(u') + p_\mathrm{\tilde k}(u') \;,
\end{align*}
where $u_0, u_0'\in\lW_0$. By subtracting the equations we get
\begin{equation*}
p_\mathrm{k}(u) - p_\mathrm{k}(u') \in \lV \cap \ker T_1 \;,
\end{equation*}
where we have used $p_\mathrm{\tilde k}(u)=p_\mathrm{\tilde k}(u')$
and $\lW_0\subseteq \lV$ (see Remark \ref{rem:Vcond}). 
However, since $\lV \cap \ker T_1 = \{0\}$, we obtain the claim. 

Furthermore, $\widetilde{S}$ is linear. Indeed, this is a simple consequence of the 
linearity of the projections. 
\smallskip

\noindent \textbf{Step III.} In this step we show that $\widetilde{S}:\ker \widetilde{T}_1\to\lH$ 
is bounded (with respect to $\|\cdot\|$ norm), for which condition (V1) will play a crucial role. 

For any $u\in\lW$ we have
\begin{equation*}
	\iscp{u}{u} = \iscp{p_\mathrm{k}(u) + p_\mathrm{\tilde k}(u)}{p_\mathrm{k}(u) + p_\mathrm{\tilde k}(u)}
	= \iscp{p_\mathrm{k}(u)}{p_\mathrm{k}(u)} + \iscp{p_\mathrm{\tilde k}(u)}{p_\mathrm{\tilde k}(u)} \;,
\end{equation*}
where in the first equality we have used $\lW_0^{[\perp]}=\lW$ (Theorem \ref{thm:abstractFO-prop}(vi)), 
while in the second $\ker \widetilde{T}_1[\perp] \ker T_1$
(see Theorem \ref{thm:abstractFO-prop}(ix)).
Furthermore, for any $\tilde\nu\in\ker\widetilde{T}_1$ it holds
\begin{equation*}
\iscp{\tilde \nu}{\tilde \nu} = \scp{T_1 \tilde\nu}{\tilde \nu} 
	= \scp{(T_1+\widetilde{T}_1)\tilde\nu}{\tilde\nu} \,,
\end{equation*} 
and similarly 
\begin{equation*}
\iscp{\nu}{\nu} = -\scp{(T_1+\widetilde{T}_1)\nu}{\nu} \;, \quad
	\nu\in\ker T_1
\end{equation*}
(here in addition we have used that $\overline{T_0+\widetilde{T}_0}$ is self-adjoint and 
$\overline{T_0+\widetilde{T}_0}|_\lW=T_1+\widetilde{T}_1$).
Thus, for any $u\in\lW$ we have
\begin{equation}\label{eq:DonW}
\iscp{u}{u} = \scp{(T_1+\widetilde{T}_1)p_\mathrm{\tilde k}(u)}{p_\mathrm{\tilde k}(u)}
	- \scp{(T_1+\widetilde{T}_1)p_\mathrm{k}(u)}{p_\mathrm{k}(u)} \;.
\end{equation}

If we restrict to $\lV$, then $\iscp{\cdot}{\cdot}$ is non-negative. Hence, for any 
$u\in\lV$ from the above equality we get
$$
\scp{(T_1+\widetilde{T}_1)p_\mathrm{k}(u)}{p_\mathrm{k}(u)} 
	\leq \scp{(T_1+\widetilde{T}_1)p_\mathrm{\tilde k}(u)}{p_\mathrm{\tilde k}(u)}
$$
(note that on both sides we have real numbers since $\overline{T_0+\widetilde{T}_0}$ is self-adjoint
and $T_1+\widetilde{T}_1\subseteq \overline{T_0+\widetilde{T}_0}$).
Using boundedness of $\overline{T_0+\widetilde{T}_0}$ from below and above and 
$T_1+\widetilde{T}_1\subseteq \overline{T_0+\widetilde{T}_0}$ (see Theorem \ref{thm:abstractFO-prop}(v),(viii)), 
we obtain
\begin{equation*}
\|p_\mathrm{k}(u)\|^2 \leq \frac{\lambda}{\mu} \|p_\mathrm{\tilde k}(u)\|^2 \;,
	\quad u\in\lV \;.
\end{equation*}
Thus, the mapping $\widetilde{S}:\ker\widetilde{T}_1\to\lH$ given by \eqref{eq:opS} is bounded (with respect to the 
norm $\|\cdot\|$).
Its adjoint $\widetilde{S}^*:\lH\to\ker\widetilde{T}_1$ is then bounded as well.

Let us denote by $\tilde{\iota}:\ker \widetilde{T}_1\hookrightarrow\lH$ the canonical embedding.
Then the operator
$$
\widetilde{Q}:=\tilde\iota^* (T_1+\widetilde{T}_1)\tilde\iota-\widetilde{S}^*(T_1+\widetilde{T}_1)\widetilde{S}
$$
is bounded and self-adjoint as an operator
$(\ker \widetilde{T}_1,\|\cdot\|)\to (\ker \widetilde{T}_1,\|\cdot\|)$.
\smallskip

\noindent\textbf{Step IV.} Returning to \eqref{eq:DonW}, but for $u\in\lV$, we have
\begin{align*}
0\leq \iscp{u}{u} &= -\scp{(T_1+\widetilde{T}_1)p_\mathrm{k}(u)}{p_\mathrm{k}(u)} 
	+\scp{(T_1+\widetilde{T}_1)p_\mathrm{\tilde k}(u)}{p_\mathrm{\tilde k}(u)} \\
&= -\scp{(T_1+\widetilde{T}_1)\widetilde{S}p_\mathrm{\tilde k}(u)}{\widetilde{S} p_\mathrm{\tilde k}(u)} 
	+\scp{(T_1+\widetilde{T}_1)\tilde\iota\, p_\mathrm{\tilde k}(u)}{\tilde\iota\, p_\mathrm{\tilde k}(u)} \\
&= -\scp{\widetilde{S}^*(T_1+\widetilde{T}_1)\widetilde{S}p_\mathrm{\tilde k}(u)}{p_\mathrm{\tilde k}(u)} 
	+\scp{\tilde\iota^*(T_1+\widetilde{T}_1)\tilde\iota\, p_\mathrm{\tilde k}(u)}{p_\mathrm{\tilde k}(u)} \\
&=\scp{\widetilde{Q} p_\mathrm{\tilde k}(u)}{p_\mathrm{\tilde k}(u)} \,.
\end{align*}
Since $p_\mathrm{\tilde k}(\lV)=\ker \widetilde{T}_1$, this implies that 
the operator $\widetilde{Q}$ is bounded, self-adjoint and positive semi-definite.
Thus, its square root $\widetilde{Q}^{1/2}$ is well-defined and bounded.

Therefore, the overall conclusion is that for any $u\in \lV$ we have
\begin{equation*}
\iscp{u}{u} = \|\widetilde{Q}^{1/2} p_\mathrm{\tilde k}(u)\|^2 \;.
\end{equation*}

Since the norm is weakly sequentially lower semicontinuous, it is left 
to see that $u_n\overset{\lW}{\dscon} u$ implies 
$\widetilde{Q}^{1/2} p_\mathrm{\tilde k}(u_n)\overset{\lH}{\dscon} \widetilde{Q}^{1/2} p_\mathrm{\tilde k}(u)$. 
However, that is an immediate consequence of the continuity of the linear operator 
$\widetilde{Q}^{1/2} p_\mathrm{\tilde k}$ with respect to the graph norm in the domain 
and the norm $\|\cdot\|$ of $\lH$ in the codomain (see Theorem \ref{thm:abstractFO-prop}(ix)).
\smallskip

\noindent\textbf{Step V.} Let us now briefly explain the argument for the second claim.
When studying $\widetilde{\lV}$, one needs to replace the operator 
$\widetilde{S}$ (given by \eqref{eq:opS}) by 
\begin{equation*}
	S(p_\mathrm{k}(u)) := p_\mathrm{\tilde k}(u) \;, \quad u\in\widetilde\lV\,,
\end{equation*}
and change $\tilde{\iota}$ to the canonical embedding of $\ker T_1$, denoted by $\iota$.
Then $Q:=\iota^* (T_1+\widetilde{T}_1)\iota-S^*(T_1+\widetilde{T}_1)S$ is again bounded, self-adjoint 
and positive semi-definite, but now on $\ker T_1$.
The claim follows from
\begin{equation*}
\iscp{u}{u} = - \|Q^{1/2} p_\mathrm{k}(u)\|^2 \;, \quad u\in\widetilde{\lV} \,.\qedhere
\end{equation*}
\end{proof}

\begin{remark}
The statement of the previous theorem holds for more general subspaces $\lV$ and $\widetilde{\lV}$.
Namely, we can take any closed subspaces $\lV$ and $\widetilde{\lV}$ in $\lW$
which satisfy condition  (V1), i.e.~$\lV$ and $\widetilde{\lV}$ are non-negative and non-positive 
with respect to $\iscp{\cdot}{\cdot}$, respectively.

This can be justified by noticing that the same proof applies also in this setting with the only 
difference that now it might be $p_\mathrm{k}(\widetilde{\lV})\neq \ker T_1$. Hence the operator
$\widetilde{S}$ will be defined on $\overline{p_\mathrm{k}(\widetilde{\lV})}$ (the closure is taken in $\lH$),
and analogously for $S$. However, this does not affect further arguments of the proof.

Alternatively, we could also use the following argument: by Zorn's 
lemma any non-negative subspace $\lV$
of $\lW$ can be extended to a maximal non-negative 
subspace $\lV_1$. This subspace by \cite[Theorem 2(b)]{ABcpde}
satisfies (V) condition, hence Theorem \ref{tm:iskp-cont}
is applicable on $\lV_1$. Since $\lV$ is a closed 
subspace of $\lV_1$, the claim follows.
\end{remark}

\begin{remark}
Since it was shown in Subsection \ref{subsec: Gkonv} that for the sequence of operators $(T_n)_n$
appearing in Definition \ref{def:Gconv} the indefinite inner product does not depend on $n$, 
the previous theorem is applicable. 
Therefore, this shows that the weakened variant of condition {\rm (K1)} is always satisfied.
\end{remark}

\section{Proof of the main theorem --- Theorem \ref{thm:Gconv}}\label{sec:proof-Gconv}

The proof follows the steps of the proof given in \cite[Theorem 5]{BVcpaa}, which was inspired by 
the original proof of Spagnolo in the case of parabolic $G$-convergence.
We present the proof in full detail as
our setting here is substantially more general than in \cite{BVcpaa} and it is important to deal with the 
precise analysis in order to verify that we can relax the equality in the condition (K1)
to an inequality, see Theorem \ref{tm:iskp-cont}. Let us start.

By the a priori estimate \eqref{eq:Fab_apriori}, $(T_n^{-1})_n$ is bounded in $\lL(\lH,\lW)$ with some bound $\gamma\geq 0$, depending on $\lambda_L$, $\alpha$ and $\beta$. Thus, by Theorem \ref{thm:seqcomp} using the separability of $\lH$, one can pass to a subsequence (for which we keep the same notation) such that $T_n^{-1}\to B \in \lL(\lH,\lW)$ in the weak operator topology, i.e., for all $f \in \lH$,  
$$
u_n\coloneqq T_n^{-1} f\dscon Bf \eqqcolon u_f
$$ weakly in $\lW$; by Theorem \ref{thm:seqcomp}, $\|B\|\leq \gamma$.
As $\lV$ is closed in $\lW$ (see Remark \ref{rem:Vcond}), 
and therefore weakly closed, we also have $B(\lH) \subseteq \lV$.


Let us define a bounded linear operator $K\in\lL(\lH)$ by $K f:=f-L u_f=f-L B f$.
Then
\begin{equation}\label{jed}
L u_n+C_n u_n= f=L u_f+K f\,,
\end{equation}
and as $L u_n \dscon L u_f$ in $\lH$ (see Remark \ref{rem:Wweak}; here we used that $\lW$ is the graph space of $L_1$), it follows
\begin{equation}\label{Cn-kvg-K}
C_n u_n\dscon K f\quad\hbox{in } \lH\,.
\end{equation}
Since $(L_1+\widetilde{L}_1)u_n\dscon (L_1+\widetilde{L}_1)u_f$ in $\lH$ (as $\overline{L_0+\widetilde{L}_0}\in\lL(\lH)$,
$L_1+\widetilde{L}_1\subseteq \overline{L_0+\widetilde{L}_0}$ and $u_n\dscon u_f$ in $\lW$), then clearly
\begin{equation}\label{Cn-kvg}
\Bigl(C_n+\frac{1}{2}(L_1+\widetilde{L}_1)\Bigr)u_n\dscon Kf +\frac{1}{2}(L_1+\widetilde{L}_1)u_f\quad\hbox{in } \lH\,.
\end{equation}
Multiplying the left equality in (\ref{jed}) by $u_n$, and the right by
$u_f$, for the real parts we get the following:
\begin{align}
\re\scp{Lu_n}{u_n}+\re\scp{C_n u_n}{u_n} &=\re\scp{f}{u_n}\,, \label{jed-un} \\
\re\scp{Lu_f}{u_f}+\re\scp{Kf}{u_f} &= \re\scp{f}{u_f}\,. \label{jed-u}
\end{align}

For any $v\in\lW$ it holds
\begin{align*}
\iscp{v} {v} &= 2\scp{L_1v}{v}-\scp{(L_1+\widetilde{L}_1)v}{v}+\Bigl(\scp{\widetilde{L}_1 v}{v}
	- \scp{v}{\widetilde{L}_1 v}\Bigr) \nonumber \\
&=2\scp{L_1v}{v}-\scp{(L_1+\widetilde{L}_1)v}{v} + 2 i \im\scp{\widetilde{L}_1 v}{v} \,.
\end{align*}
Thus, for the real parts we have
\begin{equation}\label{iscp-v}
\re \scp{L_1v}{v} = \frac{1}{2}\iscp{v}{v} + \frac{1}{2}\scp{(L_1+\widetilde{L}_1)v}{v} \,,
\end{equation}
where we have used that both terms on the right-hand side are real (the indefinite inner product is a Hermitian form
and $\overline{L_0+\widetilde{L}_0}$ is self-adjoint). Using this identity we can rewrite (\ref{jed-un}) as
\begin{equation*}
\frac{1}{2}\iscp{u_n}{u_n} = \re\scp{f}{u_n} - \re\scpp{\Bigl(C_n + \frac{1}{2}(L_1+\widetilde{L}_1)\Bigr)u_n}{u_n}
\end{equation*}
(note that $Lu_n=L_1u_n$ since $L\subseteq L_1$).
From Theorem \ref{tm:iskp-cont} it now follows
\begin{equation*}
\frac{1}{2}\iscp{u_f}{u_f} \leq \liminf_{n\to\infty}\left(\frac{1}{2}\iscp{u_n}{u_n}\right) 
	= \re\scp{f}{u_f} - \limsup_{n\to\infty}\,\re\scpp{\Bigl(C_n + \frac{1}{2}(L_1+\widetilde{L}_1)\Bigr)u_n}{u_n}\,.
\end{equation*}
By (\ref{iscp-v}), the left-hand side equals $\re\scp{L_1 u_f}{u_f}-\frac{1}{2}\scp{(L_1+\widetilde{L}_1)u_f}{u_f}$, 
while the first term on the right-hand side is given in (\ref{jed-u}).
Therefore, we obtain
\begin{equation}\label{limsup-ineq}
\limsup_{n\to\infty}\,\re\scpp{\Bigl(C_n + \frac{1}{2}(L_1+\widetilde{L}_1)\Bigr)u_n}{u_n} \leq 
	\re\scpp{Kf + \frac{1}{2}(L_1+\widetilde{L}_1)u_f}{u_f}\,.
\end{equation}

Let us show that $B$ is injective: if $Bf=u_f = 0$ for some $f\in \lH$, then (\ref{limsup-ineq}) reads
\begin{equation*}
\limsup_{n\to\infty}\,\re\scpp{\Bigl(C_n + \frac{1}{2}(L_1+\widetilde{L}_1)\Bigr)u_n}{u_n} \le 0\,,
\end{equation*}
which, together with \eqref{eq:Fab_bdd}, the second inequality in the definition of $\fabs$, implies
the strong convergence $(C_n + \frac{1}{2}(L_1+\widetilde{L}_1))u_n\str 0$ in $\lH$. 
Now from $(L_1+\widetilde{L}_1)u_n\dscon  (L_1+\widetilde{L}_1)u_f=0$ in $\lH$ we obtain $C_n u_n\dscon 0$ in $\lH$, and thus, by (\ref{Cn-kvg-K}), $K f =0$. 
Finally, from the second equality in (\ref{jed}) we are able to conclude that $f=L u_f+Kf=0+0=0$.

The injectivity of $B$ enables us to define a linear operator $C$ on the image of $B$ by
$$
C(B f)=Cu_f:=K f\in\lH.
$$
Using (\ref{Cn-kvg}), \eqref{eq:Fab_bdd} (the second inequality in the definition
of $\fabs$), and (\ref{limsup-ineq}), respectively, we have
\begin{equation*}
\begin{aligned}
\frac{1}{\beta}\Bigl\|\Bigl(C + \frac{1}{2}(L_1+\widetilde{L}_1)\Bigr)u_f\Bigr\|^2
&\le \liminf_{n\to\infty} \frac{1}{\beta}\Bigl\|\Bigl(C_n + \frac{1}{2}(L_1+\widetilde{L}_1)\Bigr)u_n\Bigr\|^2 \\
&\le \liminf_{n\to\infty}\, \re\scpp{\Bigl(C_n + \frac{1}{2}(L_1+\widetilde{L}_1)\Bigr)u_n}{u_n}  \\
&\le\re\scpp{\Bigl(C + \frac{1}{2}(L_1+\widetilde{L}_1)\Bigr)u_f}{u_f}\,,
\end{aligned}
\end{equation*}
which shows that $C$ satisfies on $B(\lH)$ the second inequality in the definition of $\fabs$. 
Thus, it is also bounded on $B(\lH)$ by the Schwarz--Cauchy--Bunjakovski inequality. 
Similarly, we obtain the first inequality, as a consequence of (\ref{limsup-ineq}):
\begin{equation*}
\begin{aligned}
\re\scpp{\Bigl(C + \frac{1}{2}(L_1+\widetilde{L}_1)\Bigr)u_f}{u_f}
	&\ge \liminf_{n\to\infty} \,\re\scpp{\Bigl(C_n + \frac{1}{2}(L_1+\widetilde{L}_1)\Bigr)u_n}{u_n} \\
	&\ge \liminf_{n\to\infty}\left(\alpha\Nor{u_n}^2\right) \ge \alpha\Nor{u_f}^2\,,
\end{aligned}
\end{equation*}
where in the second inequality we have used that $C_n\in\fabs$, i.e.~(\ref{eq:Fab_coercive}).

It is left to prove that $B(\lH)$ is dense in $\lH$, and consequently that $C$
can be uniquely extended by continuity to the whole of $\lH$ and this extended operator will then belong to $\fabs$.
For the proof of $\lH=\overline{B(\lH)}$, let $f\in B(\lH)^\bot$ (the orthogonal complement is considered with respect to the standard inner product $\scp{\cdot}{\cdot}$ of the Hilbert
space $\lH$). 
In particular, the equality $\scp{f}{u_f}=0$
holds, which together with (\ref{jed-u}) implies
$$
\re\Skp{L_1 u_f}{u_f}+\re\Skp{Kf}{u_f}=0\,.
$$
If we express the first term by using (\ref{iscp-v}), we get
$$
\frac{1}{2}\iscp{u_f}{u_f}+\re\scpp{Kf+ \frac{1}{2}(L_1+\widetilde{L}_1)u_f}{u_f}=0\,.
$$
Since $u_f\in\lV$, and thus $\iscp{u_f}{u_f}\geq 0$, we have $\re\Skp{Kf+ \frac{1}{2}(L_1+\widetilde{L}_1)u_f}{u_f}\leq 0$. 
From (\ref{limsup-ineq}) and \eqref{eq:Fab_coercive} (the first inequality in the definition of $\fabs$) 
we have $\limsup_{n\to\infty} \Nor{u_n}^2 \le 0$, which gives us the strong convergence $u_n\str 0$ in $\lH$. 
Therefore, $u_f = Bf=0$, and the injectivity of $B$ gives us $f=0$. We deduce that $B(\lH)^\perp = \{0\}$. 
Hence, $B(\lH)$ is dense in $\lH$, which concludes the proof.

\section{G-convergence of classical Friedrichs operators}\label{sec:classFOs}

In this section we revisit and enhance the results of homogenisation for Friedrichs systems that were presented in \cite{BVcpaa}. Let us recall the setting of classical partial differential operators: let $d,r\in \N$, and assume that $\Omega \subseteq \Rd$ is an open set with Lipschitz boundary $\Gamma$. 
Furthermore, let us suppose that the matrix-valued 
functions $\mA_k \in {\mathrm W}^{1,\infty}(\Omega;\Mat r\R)$, $k\in \{1,2,\dots,d\}$,  satisfy:
$$
\mA_k \,\hbox{is symmetric:}\,\mA_k= \mA_k^\top\,.
\eqno(\hbox{F1})
$$
For $\lD:=\Cbc{\Omega ;\Rr}, \lH:=\Ld{\Omega ;\Rr}$ we define the 
operators
$L_0, \widetilde L_0 : \lD\str \lH$ by the 
formul\ae
\begin{equation}\label{Lnula}
\begin{aligned}
L_0\vu :=& \sum_{k=1}^d \partial_k(\mA_k \vu)\,, \\
\widetilde L_0\vu :=& -\sum_{k=1}^d \partial_k(\mA_k \vu) +
(\sum_{k=1}^d \partial_k\mA_k)\vu \,,
\end{aligned}
\end{equation}
where $\partial_k$ stands for the {\sl classical} partial derivative. Then one can easily see that
$L_0$ and $\widetilde L_0$ satisfy (T1) and (T2). Indeed, integration by parts and (F1) imply (T1), while (T2)
follows from the regularity assumptions on $\mA_k$. More precisely, $L_0+\widetilde L_0$ coincides with the multiplication operator $\vu\mapsto  (\sum_{k=1}^d \partial_k\mA_k)\vu$ which is a bounded linear operator $\lH\rightarrow \lH$, as $\sum_{k=1}^d \partial_k\mA_k \in \Lb{\Omega;\Mat r\R}$.

Therefore, the statements i) to vi) of Theorem \ref{thm:abstractFO-prop} are valid, and the corresponding extensions $L_1$ and $\widetilde L_1$ can be represented by the same formul\ae{}
as in (\ref{Lnula}), with the distributional derivatives instead of the classical ones \cite{ABcpde}.

The graph space $\lW$ is here given by
$$
\lW=\Bigl\{\vu \in \Ld{\Omega;\Rr} : \sum_{k=1}^d \partial_k(\mA_k \vu)  \in \Ld{\Omega;\Rr}\Bigr\} \,,
$$
and the space
$\Cbc{\Rd;\Rr}$ is 
dense in $\lW$. More properties of  graph spaces of the first-order partial differential operators can be found in \cite{ABmc, MJensen}.

If we denote by $\mnu =(\nu_1,\nu_2,\dots,\nu_d)\in \Lb{\Gamma;\Rd}$
the unit outward normal on $\Gamma$, and define a matrix field on $\Gamma$ by
$$
\mA_\mnu:=\sum_{k=1}^d\nu_k\mA_k\,,
$$
then for $\vu,\vv\in \Cbc{\Rd;\Rr}$ the indefinite inner product is given by \cite{ABmc}
\begin{equation}\label{boundary-op}
\iscp \vu \vv =\int_{\Gamma}\mA_\mnu(\mx)\vu\rest{\Gamma}(\mx)\cdot
\vv\rest{\Gamma}(\mx) dS(\mx)\,.
\end{equation}

\begin{remark}
When checking conditions (T1) and (T2) above, 
one can notice that we only used that the sum $\sum_{k=1}^d\partial_k \mA_k$
is bounded, and not that all first order derivatives of each $\mA_k$ are bounded.
 
Thus, one can relax this regularity assumption and ask merely for 
$\mA_k \in {\mathrm L}^{\infty}(\Omega;\Mat r\R)$, $k\in \{1,2,\dots,d\}$,
together with 
$\sum_{k=1}^d\partial_k \mA_k \in {\mathrm L}^{\infty}(\Omega;\Mat r\R)$
(see \cite[Section 5]{EGC}).
However, in order to keep the representation \eqref{boundary-op} some additional 
assumptions are required (either on the domain or on the coefficients $\mA_k$).
\end{remark}

Notice that although the pair $(L_0, \widetilde L_0)$ satisfies (T1) and (T2), it does not need to satisfy (T3). However, we can apply the general framework of Subsection \ref{subsec: Gkonv} and conclude that for an arbitrary $C\in \fab$, i.e.~$C\in\lL(\lH)$ such that 
\begin{align*}
\scpp{\Bigl(C+\frac{1}{2}\sum_{k=1}^d \partial_k\mA_k\Bigr)\varphi}{\varphi} 
	\geq \max \Bigl\{\alpha \|\varphi\|^2, \frac{1}{\beta} \Bigl\|\Bigl(C+\frac{1}{2}\sum_{k=1}^d \partial_k{\bf A}_k\Bigr)\varphi\Bigr\|^2\Bigr\} 
	\,, \quad \varphi\in \lH\,,
\end{align*}
the pair of operators
\begin{equation}
T_0:=L_0+C \qquad \hbox{and} \qquad \widetilde{T}_0:=\widetilde{L}_0+C^* 
\end{equation}
will satisfy all properties (T1)--(T3). Actually, all statements of Subsection \ref{subsec: Gkonv} are valid, including the compactness of G-convergence, 
i.e.~Theorem \ref{thm:Gconv} (note that $\lH=\Ld{\Omega ;\Rr}$ is a separable Hilbert space).

\begin{remark}
The setting that we have presented in this section is actually wider than the classical Friedrichs setting in which for the continuous linear operator $C$ only the multiplication operator $(C\vu)(\cdot) = \mC(\cdot)\vu(\cdot)$, for bounded matrix-valued function $\mC\in \Lb{\Omega; \Mat{r}{\R}}$, was considered. If this is the case, then the operator $\Op$
is called {\sl the Friedrichs operator} or {\sl the symmetric positive operator},
and the corresponding first-order system of partial differential
equations $\Op \vu = \vf$, for a given function 
$\vf \in \Ld{\Omega;\R^r}$, is called
{\sl the Friedrichs system} or {\sl the symmetric positive system}. 

\end{remark}

\begin{remark}
	We have presented here Friedrichs systems for real-valued functions. The same can be done for complex-valued functions \cite{ABCE}.
\end{remark}
\begin{remark}
	As already mentioned, the compactness of G-convergence was shown in \cite{BVcpaa} for a specific choice of $L_0$ for which all matrices $\mA_k$ are constant, and under the assumption (K1). Thus, our framework generalises those results in several directions: it deals with abstract Friedrichs operators, allows non-constant matrices in the classical Friedrichs setting, treats the complex case as well, and (K1) is proved to be superfluous.
\end{remark}

\begin{example}
	
One of the examples that were studied in \cite{BVcpaa} was the first-order equation that exhibits memory effects in homogenisation. There the authors showed that their condition (K1) is not satisfied, and thus G-compactness cannot be obtained through the original result \cite{BVcpaa}. The present setting for Friedrichs systems established here is sufficient. It seems that nonlocal effects via homogenisation are a sign hinting at the necessity for the stronger theory developed here. 

The probably simplest model problem where such nonlocal effects occur has been provided in \cite{Tbis} and is presented next:
\begin{equation}\label{ex:jed}
\left\{
\begin{aligned}
		&\partial_t u(x,t) + c(x)u(x,t) = f(x,t)\,,\qquad (x,t)\in \Omega =\Sigma\times\oi0T \,,\\
		&u(x,0)=0\,,\quad x\in \Sigma\,.\cr
\end{aligned}
	\right.
\end{equation}
Here, $\langle\cdot, \cdot\rangle$ stands for an open interval, and, for simplicity, we take $\Sigma =\oi01$. We additionally assume $T>0$, $f\in\Ld\Omega$ and
	$\beta\ge c \ge \alpha >0$ a.e.~on $\Sigma$, for some constants $\alpha$ and $\beta$. 
	
If we denote $\lH = \Ld{\Omega}, \lD=\Cbc{\Omega}$, $L_0 u:=\partial_t u, \widetilde L_0 u := -\partial_t u$ and if $C\in \fabs$ is the multiplication operator $(Cu)(x,t) = c(x)u(x,t)$, $(x,t)\in \Omega$, then operators $T_0 :=L_0 + C$ and $\widetilde T_0:= \widetilde L_0 + C^\ast$ form a joint pair of Friedrichs operators, while $(L_0, \widetilde L_0)$ satisfy (T1) and (T2).
Note that our matrix-valued functions are actually now being scalars: $\mA_1 =0, \mA_2=1$. 

The graph space is
\begin{align*}
	\lW&=\bigl\{u\in \Ld\Omega : \partial_t u\in \Ld\Omega\bigr\}\\
	&= \bigl\{u\in \Ld{0,T;\Ld\Sigma} : \partial_t u\in \Ld{0,T;\Ld\Sigma}\bigr\}\,,
\end{align*}
and it is continuously embedded (see \cite{LM}) in ${\rm C}({0,T;\Ld\Sigma})$, while the indefinite inner product reads
\begin{equation}\label{ex:iscp}
	\iscp uv =
	\int_0^1 u(x, T)v(x, T) \,d x-\int_0^1 u(x, 0)v(x, 0) \,dx\,,\qquad u,v\in \lW.
\end{equation}
The subspaces $\lV$ and $\widetilde \lV$ that correspond to the initial condition in (\ref{ex:jed}) are given by
\begin{align*}
	\lV&=\{u\in \lW : u(\cdot, 0)=0\}\,,\\
	\widetilde \lV&=\{v\in \lW : v(\cdot, T)=0\}\,,
\end{align*}
	and one can easily verify that they satisfy condition (V).
	
Thus, all prerequisites for studying G-convergence for this initial problem in our framework are fulfilled. Since $L_0 + \widetilde L_0 =0$, we have that $\fabs$ is the set of those operators 
$C\in\lL(\lH)$ satisfying
\begin{align*}
\scp{C\varphi}{\varphi} &\geq \alpha \|\varphi\|^2 \,,\\
\scp{C\varphi}{\varphi} &\geq \frac{1}{\beta} \|C\varphi\|^2 \,,
\end{align*}
for any $\varphi\in\lH$.

In \cite{Tbis}, the author studies a bounded sequence $(c_n)_n$ of scalar functions as multiplication operators $(C_n)_n$. As an ansatz, a limit operator is guessed that admits the form of a convolution type nonlocality. Note that in the concluding part of \cite{Tbis}, the author argues, why such an equation is to be expected. For the time-horizon $\R$, in \cite{W12_HO,Waurick2014,STW22}, it is argued that such limiting equations are to be expected by means of operator-theoretic arguments. This, however, hinges on the particular realisation of the time-derivative. Here, we may apply our general G-compactness result.

By Theorem \ref{thm:Gconv}  for any sequence $C_n \in \fabs$ the corresponding (sub)sequence of isomorphisms
$$
T_n:=\partial_t+C_n:\lV\to\lH
$$
G-converges to an isomorphism $T:=\partial_t+C:\lV\to\lH$, for some $C\in\fabs$. Clearly, we can take that every $C_n$ is a multiplication operator of the form $(C_nu)(x,t) = c_n(x)u(x,t)$, $(x,t)\in \Omega$, for some $c_n\in \Lb{\Sigma}$ satisfying $\beta\ge c_n \ge \alpha >0$ a.e.~on $\Sigma$, however note that the operator $C$ in the corresponding G-limit needs not be a multiplication operator, see, e.g., \cite{Tbis} or \cite[Example 13.3.3]{STW22}.

There are several other interesting settings of homogenisation that are related to the equation in (\ref{ex:jed}). For all of them we can apply our setting and Theorem \ref{thm:Gconv}. We shall describe them here, by emphasising the main differences with the above initial problem.

The first possibility is to change the initial condition and instead of $u(\cdot,0)=0$ to take the \emph{periodic initial condition}: $u(\cdot,0)=u(\cdot,T)$ a.e.~on $\Sigma$. In order to realise this initial condition we can take
\begin{equation*}
\lV=\widetilde \lV=\{u\in \lW : u(\cdot, 0)=u(\cdot,T)\}\,.
\end{equation*}

The second possibility is to consider an infinite time interval, i.e.~to put $T=+\infty$. The only difference is that the first term in the expression (\ref{ex:iscp})  for indefinite inner product vanishes. To model the same initial condition from \eqref{ex:jed}, we take
\begin{align*}
\lV&=\{u\in \lW : u(\cdot, 0)=0\}=\lW_0\,,\\
\widetilde \lV&=\lW\,.
\end{align*}

The third possibility is to consider only the equation on $\Omega =\Sigma\times\R$ (without any initial or boundary condition). Now we have that $\iscp uv = 0$, for any $u,v\in\lW$, which is a trivial situation when $\lW_0=\lW$, and thus we can take $\lV=\widetilde{\lV} = \lW$.

\end{example}

\section{Acknowledgements}
K.B.~acknowledges funding by the Croatian Science Foundation under the grant IP-2022-10-5181 (HOMeoS).
M.E.~acknowledges funding by the Croatian Science Foundation under the project UIP-2017-05-7249 (MANDphy).
M.W.~cordially thanks the Department of Mathematics of the University of Osijek for the hospitality  extended to him in the January of 2023.

\section*{Data Availability}

 Data sharing not applicable to this article as no datasets were generated or analysed during the current study.

\end{document}